\documentclass[twoside,11pt,reqno]{amsart}

\usepackage{amsmath,amsthm,amssymb,amstext,amsfonts,amscd,mathrsfs}
\usepackage{graphicx}
\usepackage{multirow}
\usepackage{float}
\usepackage{url}
\usepackage[sc]{mathpazo}
\usepackage{physics}
\newtheorem{definition}{Definition}[section]
\newtheorem{theorem}{Theorem}[section]

\newtheorem{example}{Example}[section]
\newtheorem{remark}{Remark}[section]

\numberwithin{equation}{section}

\begin{document}
	
	\title[Numerical and Graphical Exploration of the...]
	{ Numerical and Graphical Exploration of the Generalized Beta-Logarithmic Matrix Function and Its Properties}

	\author[{\bf N. U. Khan, R. Sk and M. Hassan}]{\bf Nabiullah Khan, Rakibul Sk and Mehbub Hassan}
	
	\bigskip
	\address{Nabiullah Khan: Department of Applied
		Mathematics, Faculty of Engineering and Technology,
		Aligarh Muslim University, Aligarh 202002, India}
	\email{nukhanmath@gmail.com, Orcid Id:0000-0003-0389-7899.}
	
	\bigskip

	\address{ Rakibul Sk: Department of Applied
		Mathematics, Faculty of Engineering and Technology,
		Aligarh Muslim University, Aligarh 202002, India}
	\email{rakibulsk375@gmail.com, Orcid Id:0009-0003-6131-1419.}
	
	\bigskip
	\address{Mehbub Hassan: Department of Applied
		Mathematics, Faculty of Engineering and Technology,
		Aligarh Muslim University, Aligarh 202002, India}
	\email{mehbubhassan406@gmail.com, Orcid Id:0009-0008-0992-4580.}

	\keywords{{Logarithmic mean, beta matrix function, beta-logarithmic matrix function. }}
	
	\subjclass[2020]{ 33B15, 15A16, 65F60, 33C05.}
	
	\begin{abstract}
		This paper investigates the generalized beta-logarithmic matrix function (GBLMF), which combines the extended beta matrix function and the logarithmic mean. The study establishes essential properties of this function, including functional relations, inequalities, finite and infinite sums, integral representations, and partial derivative formulas. Theoretical results are accompanied by numerical examples and graphical representations to demonstrate the behavior of the new matrix function. Additionally, a comparison with classical and previously studied beta matrix functions is presented to highlight the differences and advantages of the generalized version. The findings offer valuable insights into the properties and applications of the extended beta-logarithmic matrix function in various mathematical and applied contexts.
	\end{abstract}
	
	\maketitle
	
	\section{\bf{Introduction and Preliminaries}}
	\noindent In parallel with the development of the extended scalar special functions, the concept of special matrix functions has gained increasing attention. The introduction of matrix versions of the gamma and beta functions was pioneered by Jódar et al. \cite{jod-cor,jode-cort}, and it has since evolved into a vibrant field of research. These matrix extensions allow us to investigate properties of classical special functions in the context of matrices, thus extending their applicability to problems involving multidimensional systems. The matrix generalizations of beta and gamma functions have demonstrated significant utility in areas such as matrix analysis, control theory, signal processing, and optimization, highlighting the importance of studying their properties in the matrix setting.
	
	Over the years, many researchers have built upon the original matrix function theory, developing various generalizations and extensions of the beta matrix function \cite{ab-ba,cekim-b,dw-sa,dwiv-saha,goy-iri,kha-agar,hus-kha,khan-rakib,verma-sahai}. These extensions not only preserve many characteristics of the classical beta function but also offer new tools for addressing more complex, multidimensional problems. The focus on matrix-based special functions has become an essential part of modern mathematical research, given their critical role in modeling and solving real-world problems in diverse scientific and engineering disciplines.
	
	In this article, we explore a new generalization of the beta-logarithmic matrix function which combines the extended beta matrix function with the logarithmic mean. Our goal is to establish a thorough understanding of its key properties, including functional relationships, inequalities, sums, integral representations, and partial derivatives. By offering a detailed analysis, we aim to elucidate the differences between this new function and previously studied beta matrix functions. Through theoretical results and numerical examples, we provide both a rigorous and practical perspective on this extended matrix function, demonstrating its potential advantages and applications.
	
	To set the stage for the discussion of the generalized beta-logarithmic matrix function, we first review some essential concepts related to special matrix functions. This will help readers better appreciate the mathematical framework underlying the new results and understand the broader context of this study within the growing body of research on matrix extensions of classical special functions.\\  
	
	Throughout the paper, let $I$ and $O$ denote the identity matrix and zero matrix in $\mathbb{C}^{k\cross k}$, respectively, and where $\mathbb{C}^{k\cross k}$ is the vector space of $k$-square matrices with complex entries. For a matrix $P\in$ $\mathbb{C}^{k\cross k}$, the spectrum is denoted by $\sigma(P)$ and it is the set of all eigenvalues of the matrix $P$. A matrix $P\in$ $\mathbb{C}^{k\cross k}$ is positive stable matrix if $\real(\mu)>0$, $\forall \mu \in \sigma(P)$. Let $\mathscr{M}_k$ be the vector space of all the $k\cross k$ hermitian positive stable matrices of order $k \in \mathbb{N}$ whose entries are in the set of complex numbers $\mathbb{C}.$ For a matrix $P\in \mathscr{M}_k$, the norm of the matrix, $P$ is defined by
	\begin{equation}\label{eq1.1}
		\|P\|= \max_{x \neq 0} \frac{\|Px\|}{\|x\|}.
	\end{equation}
	The infinity norm of a square matrix is the maximum of the absolute row sums. For the matrix $P\in \mathscr{M}_k$, $\|P\|_\infty$ is given by
	\begin{equation}\label{eq1.2}
		\|P\|_\infty=\max_{1\leq i\leq k} \sum_{j=1}^{k} |\alpha_{ij}|.
	\end{equation}
	
	In 1997, Jodar and Cortes introduced matrix parameters in the classical Euler beta function. If $P,Q\in \mathscr{M}_k$, such that $PQ=QP$, the classical beta matrix function (CBMF) $\mathfrak{B}(P,Q)$ is well defined as follows \cite{jod-cor}:
	\begin{equation}\label{eq1.3}
		\mathfrak{B}(P,Q)=\int_{0}^{1} {x}^{P-I}(1-x)^{Q-I} ~dx=\mathfrak{B}(Q,P).
	\end{equation} 
	
	For $P,Q$ and $R \in \mathscr{M}_k,$ the extended beta function involving matrix argument is defined by Abdalla and Bakhet \cite{ab-ba} as follows:
	\begin{equation}\label{eq1.4}
		\mathfrak{B}^{R}(P,Q)=\int_{0}^{1}x^{P-I}(1-x)^{Q-I}\exp\left(-\frac{R}{x(1-x)}\right)~dx.
	\end{equation}
	
	The logarithmic mean  which lies between the geometric mean and arithmetic mean for $a,b>0$ is given by 
	\begin{equation}\label{eq1.5}
		\mathfrak{L}(a,b)=\int_{0}^{1} a^{1-x}~b^x~dx=\begin{cases}
			a\qquad\qquad\qquad\qquad\qquad a=b\\\frac{a-b}{\log(a)-\log(b)}\qquad\qquad\qquad a\neq b
		\end{cases}.
	\end{equation}

	For $P,Q$ and $R \in \mathscr{M}_k,$ , the extended beta-logarithmic matrix function (EBLMF) is defined by Alqarni \cite{alq} as follows:
	\begin{equation}\label{eq1.6}
		\mathfrak{BL}^{R}(a,b;P,Q)=\int_{0}^{1}a^{1-x}~b^{x}x^{P-I}(1-x)^{Q-I}\exp\left(-\frac{R}{x(1-x)}\right)~dx.
	\end{equation}
	
	Later in 2024, Khan et al. \cite{khan-sk} investigated a new kind of beta function using two parameter Mittag-Leffler function as follows:
	\begin{equation}\label{eq1.7}
		\mathfrak{B}_{(\phi,\psi)}^{(r,s,\eta,\xi)}(p,q)=\int_{0}^{1}x^{p-1}(1-x)^{q-1}\mathcal{E}_{(\phi,\psi)}\left(-\frac{r}{x^{\eta}}\right)\mathcal{E}_{(\phi,\psi)}\left(-\frac{s}{(1-x)^{\xi}}\right)~dx~,
	\end{equation}
	$$\left(\real(p)>0, \real(q)>0; \phi,\psi,\eta,\xi \in \mathbb{R}^{+}; r\ge0,~ s\ge0\right),$$
	
	where $\mathcal{E}_{\phi,\psi}(z)$ is the Mittag-Leffler function defined as \cite{a-wiman} :
	\begin{equation}\label{eq1.8}
		\mathcal{E}_{\phi,\psi}(z)=\sum_{k=0}^{\infty}\frac{z^k}{\Gamma(\phi k+\psi)},\qquad z\in \mathbb{C}.
	\end{equation}\\
	
	They also investigated a new kind of beta-logarithmic function using \eqref{eq1.5} and \eqref{eq1.7} as follows:
	$$\mathfrak{BL}_{(\phi,\psi)}^{(r,s,\eta,\xi)}(a,b;p,q)$$
	\begin{equation}\label{eq1.9}
		=\int_{0}^{1}a^{1-x}b^{x}x^{p-1}(1-x)^{q-1}\mathcal{E}_{(\phi,\psi)}\left(-\frac{r}{x^{\eta}}\right)\mathcal{E}_{(\phi,\psi)}\left(-\frac{s}{(1-x)^{\xi}}\right)~dx~,
	\end{equation}
	$$\left(\phi,\psi,\eta,\xi,a,b\in \mathbb{R}^{+};r,s\ge0;\real(p)>0,\real(q)>0\right).$$\\
	
	Inspired and motivated by the above certain extensions of the special matrix functions, we introduce a new matrix setting in the beta function and beta-logarithmic function defined in \eqref{eq1.7} and \eqref{eq1.9} respectively. We discussed some important properties of these extended matrix functions and investigated some numerical and graphical exploration of the extended beta-logarithmic matrix function to show the comparison with previously investigated results.

	\vspace{0.35cm}
	\section{\bf A Generalized Beta and Beta-Logarithmic Matrix Function}
	\bigskip
	
	\begin{definition}
		Let $ P, Q, R $ and $S \in \mathscr{M}_k,$ the generalized beta matrix function (GBMF) defined as:
		
		\begin{equation}\label{eq2.1}
			\mathfrak{B}_{(\phi,\psi)}^{(R,S,\eta,\xi)}(P,Q)=\int_{0}^{1}x^{P-I}(1-x)^{Q-I}\mathcal{E}_{(\phi,\psi)}\left(-\frac{R}{x^{\eta}}\right)\mathcal{E}_{(\phi,\psi)}\left(-\frac{S}{(1-x)^{\xi}}\right)~dx~,
		\end{equation}
		
		where the two-parameter Mittag-Leffler matrix function is defined \cite{gar-pop} as:
		\begin{equation*}
			\mathcal{E}_{\phi, \psi}(P)=\sum_{k=0}^{\infty}\frac{P^k}{\Gamma(\phi k+\psi)},
		\end{equation*}
		
		$$\left(\phi,\psi,\eta,\xi \in \mathbb{R}^{+} \right).$$
		
	\end{definition}

	\subsection*{Remark 2.1}
	a. If we put $\phi=\psi=\eta=\xi=1$ and $R=S$, then \eqref{eq2.1} reduces to \eqref{eq1.4}, which is the extended beta matrix function defined in \cite{ab-ba}:
	\begin{equation*}
		\mathfrak{B}_{(1,1)}^{(R,R,1,1)}(P,Q)=\mathfrak{B}^{R}(P,Q) .
	\end{equation*}

	b. If we choose  $\phi=\psi=\eta=\xi=1$ and $R=S=0_k$ in \eqref{eq2.1}, we get \eqref{eq1.3}, which is the matrix form of classical Euler beta function defined in \cite{jod-cor}:
	\begin{equation*}
		\mathfrak{B}_{(1,1)}^{(0,0,1,1)}(P,Q)=\mathfrak{B}(P,Q).
	\end{equation*}\\
	Now, using the beta function \eqref{eq2.1} and logarithmic mean \eqref{eq1.5}, we introduce the matrix setting of \eqref{eq1.9} and defined the generalized beta-logarithmic matrix function (GBLMF) in the following definition.
	\begin{definition}
		Let $ P, Q, R $ and $S \in \mathscr{M}_k,$ the generalized beta-logarithmic matrix function (GBLMF) defined as:
		$$\mathfrak{BL}_{(\phi,\psi)}^{(R,S,\eta,\xi)}(a,b;P,Q)$$
		\begin{equation}\label{eq2.2}
			=\int_{0}^{1}a^{1-x}b^{x}x^{P-I}(1-x)^{Q-I}\mathcal{E}_{(\phi,\psi)}\left(-\frac{R}{x^{\eta}}\right)\mathcal{E}_{(\phi,\psi)}\left(-\frac{S}{(1-x)^{\xi}}\right)~dx~,
		\end{equation}
	\end{definition}
	
	$$\left(\phi,\psi,\eta,\xi,a,b\in \mathbb{R}^{+}\right).$$
	
	\subsection*{Remark 2.2}
	a. If we put $\phi=\psi=\eta=\xi=1$ and $R=S$, then \eqref{eq2.2} reduces to \eqref{eq1.6}, which is the extended beta-logarithmic matrix function defined in \cite{alq}:
	\begin{equation*}
		\mathfrak{BL}_{(1,1)}^{(R,R,1,1)}(a,b;P,Q)=\mathfrak{BL}^{R}(a,b;P,Q) .
	\end{equation*}
	b. Substituting $a=b=1$ in \eqref{eq2.2}, we get the extended beta matrix function \eqref{eq2.1}:
	\begin{equation*}
		\mathfrak{BL}_{(\phi,\psi)}^{(R,S,\eta,\xi)}(1,1;P,Q)=\mathfrak{B}_{(\phi,\psi)}^{(R,S,\eta,\xi)}(P,Q) .
	\end{equation*}
	c. If we take $a=b=\phi=\psi=\eta=\xi=1$ ; $R=S$ in \eqref{eq2.2} we get \eqref{eq1.4}, the extended beta matrix function introduced by Abdalla et al. \cite{ab-ba} :
	\begin{equation*}
		\mathfrak{BL}_{(1,1)}^{(R,R,1,1)}(1,1;P,Q)=\mathfrak{B}^{R}(P,Q).
	\end{equation*}
	d. If we choose  $a=b=\phi=\psi=\eta=\xi=1$ and $A=B=0_k$ in \eqref{eq2.2}, we get \eqref{eq1.3}, which is the matrix form of classical Euler beta function defined in \cite{jod-cor}:
	\begin{equation*}
		\mathfrak{BL}_{(1,1)}^{(0,0,1,1)}(1,1;P,Q)=\mathfrak{B}(P,Q).
	\end{equation*}
	e. If we set $P=Q=I_k$; $\phi=\psi=\eta=\xi=1$ ; $R=S=0_k$ in \eqref{eq2.2}, we obtain logarithmic mean \eqref{eq1.5}:
	\begin{equation*}
		\mathfrak{BL}_{(1,1)}^{(0,0,1,1)}(a,b;I,I)=\mathfrak{L}(a,b).
	\end{equation*}\\
	
	\vspace{0.35cm}
	\section{\bf Properties of Beta-Logarithmic Matrix Function}
	In this section, we obtain some different properties and representations of the generalized beta-logarithmic matrix function \eqref{eq2.2}.
	\subsection*{Proposition 3.1.} For any $\phi,\psi,a,b >0;$ $P,Q,R,S\in \mathscr{M}_k$, the GBLMF \eqref{eq2.2} holds the following properties :
	\begin{equation}\label{eq3.1}
		\mathfrak{BL}_{(\phi,\psi)}^{(R,S,\eta,\xi)}(a,b;P,Q)=\mathfrak{BL}_{(\phi,\psi)}^{(S,R,\xi,\eta)}(b,a;Q,P),
	\end{equation}
	\begin{equation}\label{eq3.2}
		\mathfrak{BL}_{(\phi,\psi)}^{(R,S,\eta,\xi)}(a,a;P,Q)=a\mathfrak{B}_{(\phi,\psi)}^{(R,S,\eta,\xi)}(P,Q),
	\end{equation}
	and
	\begin{equation}\label{eq3.3}
		\mathfrak{BL}_{(\phi,\psi)}^{(R,S,\eta,\xi)}(\xi a,\xi b;P,Q)=\xi\mathfrak{BL}_{(\phi,\psi)}^{(R,S,\eta,\xi)}(a,b;P,Q).
	\end{equation}
	\begin{proof}
		By changing the variable $x$ into $(1-y)$ in \eqref{eq2.2} and after simplification, we get the result \eqref{eq3.1}. Also, the assertions \eqref{eq3.2} and \eqref{eq3.3}  obtained by simple calculation in \eqref{eq2.2}.
	\end{proof}
	
	\subsection*{Proposition 3.2.}For $\phi,\psi,a,b>0$; $P,Q,R,S\in \mathscr{M}_k$,  the following assertions holds true for the GBLMF \eqref{eq2.2} :
	\begin{equation}\label{eq3.4}
		\mathfrak{BL}_{(\phi,\psi)}^{(R,S,\eta,\xi)}(a,b;P+I,Q)+\mathfrak{BL}_{(\phi,\psi)}^{(R,S,\eta,\xi)}(a,b;P,Q+I)=\mathfrak{BL}_{(\phi,\psi)}^{(R,S,\eta,\xi)}(a,b;P,Q).
	\end{equation}
	\begin{proof}
		Apply \eqref{eq2.2} to the left side of \eqref{eq3.4} and simplify  we get right side of required result \eqref{eq3.4}.
	\end{proof}
	\subsection*{Corollary 3.1} If we substitute $a=b=1$ in \eqref{eq3.4}, we obtain the following result :
	\begin{equation}\label{eq3.5}
		\mathfrak{B}_{(\phi,\psi)}^{(R,S,\eta,\xi)}(P+I,Q)+\mathfrak{B}_{(\phi,\psi)}^{(R,S,\eta,\xi)}(P,Q+I)=\mathfrak{B}_{(\phi,\psi)}^{(R,S,\eta,\xi)}(P,Q).
	\end{equation}
	\subsection*{Corollary 3.2} If we substitute $R=S$; $\phi=\psi=\eta=\xi=1$ in \eqref{eq3.4}, we obtain the following result :
	\begin{equation}\label{eq3.6}
		\mathfrak{BL}^{R}(a,b;P+I,Q)+\mathfrak{BL}^{R}(a,b;P,Q+I)=\mathfrak{BL}^{R}(a,b;P,Q).
	\end{equation}
	\subsection*{Corollary 3.3} If we substitute $R=S=0_k$; $\phi=\psi=\eta=\xi=1$ in \eqref{eq3.4}, we obtain the following result :
	\begin{equation}\label{eq3.7}
		\mathfrak{BL}(a,b;P+I,Q)+\mathfrak{BL}(a,b;P,Q+I)=\mathfrak{BL}(a,b;P,Q).
	\end{equation}
	\subsection*{Corollary 3.4} If we substitute $R=S=0_k$; $\phi=\psi=\eta=\xi=1$; $a=b=1$ in \eqref{eq3.4}, we obtain the following result :
	\begin{equation}\label{eq3.8}
		\mathfrak{B}(P+I,Q)+\mathfrak{B}(P,Q+I)=\mathfrak{B}(P,Q).
	\end{equation}\\
	The following theorem delivers a bound of the function given in \eqref{eq2.2}.
	\begin{theorem}\label{thm3.1}
		For any $\phi,\psi,a,b>0$; $P,Q,R,S \in \mathscr{M}_k$, the beta-logarithmic matrix function holds the following inequality :

		\begin{equation}\label{eq3.9}
			\begin{aligned}
				\min(a,b)\left\Vert\mathfrak{B}_{(\phi,\psi)}^{(R,S,\eta,\xi)}(P,Q)\right\Vert\le\left\Vert \mathfrak{BL}_{(\phi,\psi)}^{(R,S,\eta,\xi)}(a,b;P,Q)\right\Vert\le 
				\\ \max(a,b) \left\Vert\mathfrak{B}_{(\phi,\psi)}^{(R,S,\eta,\xi)}(P,Q)\right\Vert.
			\end{aligned}
		\end{equation}

		\begin{proof}
			Here, using the well known inequality \\
			
			$\min(a,b)\le\sqrt{ab}\le\mathfrak{L}(a,b)\le(\frac{a+b}{2})\le\max(a,b)$
			and $\left\Vert\mathfrak{B}_{(\phi,\psi)}^{(R,S,\eta,\xi)}(P,Q)\right\Vert>0$\\
			
			we obtain the inequality
			\begin{equation}\label{eq3.10}
				\min(a,b)\left\Vert\mathfrak{B}_{(\phi,\psi)}^{(R,S,\eta,\xi)}(P,Q)\right\Vert\le\left\Vert \mathfrak{BL}_{(\phi,\psi)}^{(R,S,\eta,\xi)}(a,b;P,Q)\right\Vert.
			\end{equation}
			Now, apply the Young's inequality
			$a^{1-x}b^{x}\le a(1-x)+bx\quad \forall~ x\in[0,1]$ in the definition \eqref{eq2.2} of right side \eqref{eq3.10}, we get
			\begin{equation*}
				\begin{aligned}
					&\min(a,b)\left\Vert\mathfrak{B}_{(\phi,\psi)}^{(R,S,\eta,\xi)}(P,Q)\right\Vert\le\left\Vert \mathfrak{BL}_{(\phi,\psi)}^{(R,S,\eta,\xi)}(a,b;P,Q)\right\Vert\\\le &a\left\Vert\mathfrak{B}_{(\phi,\psi)}^{(R,S,\eta,\xi)}(P,Q+I)\right\Vert+b \left\Vert\mathfrak{B}_{(\phi,\psi)}^{(R,S,\eta,\xi)}(P+I,Q)\right\Vert\\
					&\le\max(a,b)\left\Vert\left[\mathfrak{B}_{(\phi,\psi)}^{(R,S,\eta,\xi)}(P,Q+I)+\mathfrak{B}_{(\phi,\psi)}^{(R,S,\eta,\xi)}(P+I,Q)\right]\right\Vert.
				\end{aligned}
			\end{equation*}
			Using \eqref{eq3.5} in above expression, we get our required result \eqref{eq3.9}.
		\end{proof}
	\end{theorem}
	\subsection*{Corollary 3.5} If we substitute $R=S$; $\phi=\psi=\eta=\xi=1$ in \eqref{eq3.9}, we obtain the following result :
	\begin{equation}\label{eq3.11}
		\min(a,b)\left\Vert\mathfrak{B}^{R}(P,Q)\right\Vert\le\left\Vert \mathfrak{BL}^{R}(a,b;P,Q)\right\Vert\\\le \max(a,b)\left\Vert\mathfrak{B}^{R}(P,Q)\right\Vert.
	\end{equation}
	\subsection*{Corollary 3.6} If we substitute $R=S=0_k$; $\phi=\psi=\eta=\xi=1$ in \eqref{eq3.9}, we obtain the following result :
	\begin{equation}\label{eq3.12}
		\min(a,b)\left\Vert\mathfrak{B}(P,Q)\right\Vert\le\left\Vert \mathfrak{BL}(a,b;P,Q)\right\Vert\\\le \max(a,b)\left\Vert\mathfrak{B}(P,Q)\right\Vert.
	\end{equation}
	\subsection*{Corollary 3.7} For $a=b=\phi=\psi=\eta=\xi=1;$ $R=S,$ let $R,P$ and $Q$ be reciprocally commutative matrices in $\mathscr{M}_k$, then \eqref{eq3.9} reduces to the following result:
	\begin{equation}\label{eq3.13}
		\left\Vert\mathfrak{B}^{R}(P,Q)\right\Vert\le exp(-4\left\Vert R\right\Vert)\left\Vert \mathfrak{B}(P,Q)\right\Vert.
	\end{equation}
	The following theorems deliver infinite sums of the function given in \eqref{eq2.2}.
	\begin{theorem}\label{thm3.2}
		Let $\phi,\psi,a,b>0$ and $P,Q,R$ and $S\in \mathscr{M}_k$, the following summation relations hold true :
		\begin{equation}\label{eq3.14}
			\mathfrak{BL}_{(\phi,\psi)}^{(R,S,\eta,\xi)}(a,b;P,Q)=\sum_{k=0}^{\infty}\mathfrak{BL}_{(\phi,\psi)}^{(R,S,\eta,\xi)}(a,b;P+I,Q+kI),
		\end{equation} and
		\begin{equation}\label{eq3.15}
			\mathfrak{BL}_{(\phi,\psi)}^{(R,S,\eta,\xi)}(a,b;P,Q)=\sum_{k=0}^{\infty}\mathfrak{BL}_{(\phi,\psi)}^{(R,S,\eta,\xi)}(a,b;P+kI,Q+I).
		\end{equation}
		\begin{proof}
			We know the power series representation 
			\begin{equation}\label{eq3.16}
				x^{-I}=exp(\ln(x^{-I}))=I\sum_{k=0}^{\infty}(1-x)^{k}\quad \forall~ x\in (0,1).
			\end{equation}
			Now, using \eqref{eq2.2}, we get
			$$\mathfrak{BL}_{(\phi,\psi)}^{(R,S,\eta,\xi)}(a,b;P,Q)$$
			\begin{equation*}
				\begin{aligned}
					=&\int_{0}^{1}a^{1-x}b^{x}x^{P-I}(1-x)^{Q-I}\mathcal{E}_{(\phi,\psi)}\left(-\frac{R}{x^{\eta}}\right)\mathcal{E}_{(\phi,\psi)}\left(-\frac{S}{(1-x)^{\xi}}\right)~dx\\
					=&\int_{0}^{1}a^{1-x}b^{x}x^{P}x^{-I}(1-x)^{Q-I}\mathcal{E}_{(\phi,\psi)}\left(-\frac{R}{x^{\eta}}\right)\mathcal{E}_{(\phi,\psi)}\left(-\frac{S}{(1-x)^{\xi}}\right)~dx,
				\end{aligned}
			\end{equation*}
			applying \eqref{eq3.16} in above expression and after simplifying we achieved our required result \eqref{eq3.14}.
			Similarly, the power series representation 
			\begin{equation}\label{eq3.17}
				(1-x)^{-I}=exp(\ln((1-x)^{-I}))=I\sum_{k=0}^{\infty}x^{k}\quad \forall~ x\in (0,1).
			\end{equation}
			Applying \eqref{eq3.17} in \eqref{eq2.2} and after simplifying we achieved our required result \eqref{eq3.15}.
		\end{proof}
	\end{theorem}
	\subsection*{Corollary 3.8}
	By setting $a=b=1$ in \eqref{eq3.14} and \eqref{eq3.15}, we can obtain the infinite sums representations of the function defined in \eqref{eq2.1} as followings:
	\begin{equation}\label{eq3.18}
		\mathfrak{B}_{(\phi,\psi)}^{(R,S,\eta,\xi)}(P,Q)=\sum_{k=0}^{\infty}\mathfrak{B}_{(\phi,\psi)}^{(R,S,\eta,\xi)}(P+I,Q+kI),
	\end{equation} and
	\begin{equation}\label{eq3.19}
		\mathfrak{B}_{(\phi,\psi)}^{(R,S,\eta,\xi)}(P,Q)=\sum_{k=0}^{\infty}\mathfrak{B}_{(\phi,\psi)}^{(R,S,\eta,\xi)}(P+kI,Q+I).
	\end{equation}
	
	\begin{theorem}\label{thm3.3}
		For any  $\phi,\psi,a,b>0$ and $P,Q,R,S \in \mathscr{M}_k$, the following representation holds true :
		\begin{equation}\label{eq3.20}
			\mathfrak{BL}_{(\phi,\psi)}^{(R,S,\eta,\xi)}(a,b;P,Q)=\sum_{k,r=0}^{\infty}\frac{\mathfrak{B}_{(\phi,\psi)}^{(R,S,\eta,\xi)}(P+kI,Q+rI)}{k!r!}\left(\log(a)\right)^{r}\left(\log(b)\right)^{k}.
		\end{equation}
		\begin{proof}
			From \eqref{eq2.2}, we get 
			$$\mathfrak{BL}_{(\phi,\psi)}^{(R,S,\eta,\xi)}(a,b;P,Q)$$
			\begin{equation*}
				=\int_{0}^{1}a^{1-x}b^{x}x^{P-I}(1-x)^{Q-I}\mathcal{E}_{(\phi,\psi)}\left(-\frac{R}{x^{\eta}}\right)\mathcal{E}_{(\phi,\psi)}\left(-\frac{S}{(1-x)^{\xi}}\right)~dx~.
			\end{equation*}
			Now, the power series expansion 
			\begin{equation}\label{eq3.21}
				a^{1-x}=\sum_{r=0}^{\infty}\frac{\left(\log(a)\right)^{r}}{r!}(1-x)^{r}
			\end{equation} 
			and 
			\begin{equation}\label{eq3.22}
				b^{x}=\sum_{k=0}^{\infty}\frac{\left(\log(b)\right)^{k}}{k!}(x)^{k},
			\end{equation}
			using \eqref{eq3.21} and \eqref{eq3.22} in the above expansion and applying \eqref{eq2.2}, we get the desired result \eqref{eq3.20}. 
		\end{proof}
	\end{theorem}
	The following theorem delivers a finite sums of the function given in \eqref{eq2.2}.
	\begin{theorem}\label{thm3.4}
			For any  $\phi,\psi,a,b,\eta,\xi>0$ and $P,Q,R,S \in \mathscr{M}_k$, the following  finite sums relation holds true :
		\begin{equation}\label{eq3.23}
			\mathfrak{BL}_{(\phi,\psi)}^{(R,S,\eta,\xi)}(a,b;P,Q)=\sum_{k=0}^{m} \binom{m}{k}\mathfrak{BL}_{(\phi,\psi)}^{(R,S,\eta,\xi)}(a,b;P+kI,Q+(m-k)I).
		\end{equation} 
	\end{theorem}
	\begin{proof}
		From \eqref{eq3.4} we have the following,
		\begin{equation*}
			\mathfrak{BL}_{(\phi,\psi)}^{(R,S,\eta,\xi)}(a,b;P,Q)=	\mathfrak{BL}_{(\phi,\psi)}^{(R,S,\eta,\xi)}(a,b;P+I,Q)+\mathfrak{BL}_{(\phi,\psi)}^{(R,S,\eta,\xi)}(a,b;P,Q+I)
		\end{equation*}
		Then, we apply the same result of \eqref{eq3.4} on each term of the right hand side of the above equation. We obtain the following,
		\begin{equation*}
			\begin{aligned}
				\mathfrak{BL}_{(\phi,\psi)}^{(R,S,\eta,\xi)}(a,b;P,Q)=\mathfrak{BL}_{(\phi,\psi)}^{(R,S,\eta,\xi)}(a,b;P+2I,Q)+\\ 2\mathfrak{BL}_{(\phi,\psi)}^{(R,S,\eta,\xi)}(a,b;P+I,Q+I)+\mathfrak{BL}_{(\phi,\psi)}^{(R,S,\eta,\xi)}(a,b;P,Q+2I).
			\end{aligned}
		\end{equation*}
		Upon continuing the same process on the right hand side of the above equation and applying mathematical induction, we obtain \eqref{eq3.23}.
	\end{proof}
	\subsection*{Corollary3.9}
	By putting $a=b=1$ in \eqref{eq3.23} we obtain finite sum representations of the function given in \eqref{eq2.1} as
	\begin{equation}
		\mathfrak{B}_{(\phi,\psi)}^{(R,S,\eta,\xi)}(P,Q)=\sum_{k=0}^{m} \binom{m}{k}\mathfrak{B}_{(\phi,\psi)}^{(R,S,\eta,\xi)}(P+kI,Q+(m-k)I).
	\end{equation}
	The following theorem gives various integral representations of the function given in \eqref{eq2.2}.
	\begin{theorem}\label{thm3.5}
		For any  $\phi,\psi,a,b,\eta,\xi>0$ and $P,Q,R,S \in \mathscr{M}_k$, the generalized beta-logarithmic matrix function \eqref{eq2.2} holds the following relations:
		
		$$\mathfrak{BL}_{(\phi,\psi)}^{(R,S,\eta,\xi)}(a,b;P, Q)$$
		\begin{equation}\label{eq3.25}
			=2a\int_{0}^{\frac{\pi}{2}}~~\left(\frac{b}{a}\right)^{\sin^{2}(\theta)}{\sin^{2P-I}(\theta)\cos^{2Q-I}(\theta)}\mathcal{E}_{\phi,\psi}(-R\csc^{2\eta}(\theta))\mathcal{E}_{\phi,\psi}(-S\sec^{2\xi}(\theta)) d{\theta},
		\end{equation}
		$$\mathfrak{BL}_{(\phi,\psi)}^{(R,S,\eta,\xi)}(a,b;P, Q)$$
		\begin{equation}\label{eq3.26}
			=\int_{0}^{\infty}~~\frac{t^{P-I}}{(1+t)^{P+Q}}a^\frac{1}{1+t}b^\frac{t}{1+t}\mathcal{E}_{\phi,\psi}(-R(1+t^{-1})^{\eta})\mathcal{E}_{\phi,\psi}(-S(1+t)^{\xi}) dt,
		\end{equation}
		and
		$$\mathfrak{BL}_{(\phi,\psi)}^{(R,S,\eta,\xi)}(a,b;P, Q)$$
		\begin{equation}\label{eq3.27}
			=\frac{2^I\sqrt{ab}}{2^{P+Q}}\int_{-1}^{1}~~(1+t)^{P-I}(1-t)^{Q-I}\left(\frac{b}{a}\right)^{\frac{t}{2}}\mathcal{E}_{\phi,\psi}(-R(\frac{1+t}{2})^{-\eta})\mathcal{E}_{\phi,\psi}(-S(\frac{1-t}{2})^{-\xi}) dt.
		\end{equation}
		\begin{proof}
			Setting $x=\sin^{2}(\theta)$ in \eqref{eq2.2} yeilds \eqref{eq3.25}. Next, replacing $x=\frac{t}{1+t}$ in \eqref{eq2.2} gives \eqref{eq3.26}. Finally, substituiting $x=\frac{1+t}{2}$ in \eqref{eq2.2} provides \eqref{eq3.27}.
		\end{proof}
	\end{theorem}
	\subsection*{Corollary 3.10}
	By setting $a=b=1$ in \eqref{eq3.25}, \eqref{eq3.26} and \eqref{eq3.27}, we obtain the integral representations of the function given in \eqref{eq2.1} as following:
	$$\mathfrak{B}_{(\phi,\psi)}^{(R,S,\eta,\xi)}(P, Q)$$
	\begin{equation}\label{eq3.28}
		=2\int_{0}^{\frac{\pi}{2}}~~{\sin^{2P-I}(\theta)\cos^{2Q-I}(\theta)}\mathcal{E}_{\phi,\psi}(-R\csc^{2\eta}(\theta))\mathcal{E}_{\phi,\psi}(-S\sec^{2\xi}(\theta)) d{\theta},
	\end{equation}
		$$\mathfrak{B}_{(\phi,\psi)}^{(R,S,\eta,\xi)}(P, Q)$$
	\begin{equation}\label{eq3.29}
		=\int_{0}^{\infty}~~\frac{t^{P-I}}{(1+t)^{P+Q}}\mathcal{E}_{\phi,\psi}(-R(1+t^{-1})^{\eta})\mathcal{E}_{\phi,\psi}(-S(1+t)^{\xi}) dt,
	\end{equation}
	and 
		$$\mathfrak{B}_{(\phi,\psi)}^{(R,S,\eta,\xi)}(P, Q)$$
	
	\begin{equation}\label{eq3.30}
		=\frac{2^I}{2^{P+Q}}\int_{-1}^{1}~~(1+t)^{P-I}(1-t)^{Q-I}\mathcal{E}_{\phi,\psi}(-R(\frac{1+t}{2})^{-\eta})\mathcal{E}_{\phi,\psi}(-S(\frac{1-t}{2})^{-\xi}) dt.
	\end{equation}\\
	
	The next theorem gives the higher-order derivative formula of the function given in \eqref{eq2.2}. 
	\begin{theorem}\label{thm3.6}
		For $\phi,\psi,\eta,\xi,a,b\in \mathbb{R}^{+};P,Q,R,S \in \mathscr{M}_k$ and $m,n \in \mathbb{N}_{0}$
		$$\frac{\partial^{(m+n)}}{\partial P^m \partial Q^n}	\left\{\mathfrak{BL}_{(\phi,\psi)}^{(R,S,\eta,\xi)}(a,b;P,Q)\right\}$$
		\begin{equation}\label{eq3.31}
			=\int_{0}^{1}\ln^{m}(x)\ln^{n}(1-x)a^{1-x}b^{x}x^{P-I}(1-x)^{Q-I}\mathcal{E}_{(\phi,\psi)}\left(\frac{-R}{x^{\eta}}\right)\mathcal{E}_{(\phi,\psi)}\left(\frac{-S}{(1-x)^{\xi}}\right)~dx.
		\end{equation}
	\end{theorem}
	\begin{proof}
		We compute the partial derivative of the equation \eqref{eq2.2} with respect to $P$ and $Q$, by applying the Leibniz rule for differentiation under integral sign, we obtain
		$$\frac{\partial^2}{\partial P \partial Q}	\left\{\mathfrak{BL}_{(\phi,\psi)}^{(R,S,\eta,\xi)}(a,b;P,Q)\right\}$$
		\begin{equation*}
			=\int_{0}^{1}\ln(x)\ln(1-x)a^{1-x}b^{x}x^{P-I}(1-x)^{Q-I}\mathcal{E}_{(\phi,\psi)}\left(\frac{-R}{x^{\eta}}\right)\mathcal{E}_{(\phi,\psi)}\left(\frac{-S}{(1-x)^{\xi}}\right)~dx.
		\end{equation*}
		By repeating the differentiation, we obtain our required result of \eqref{eq3.31}.
	\end{proof}
	\subsection*{Corollary 3.11} 
	By putting $a=b=1$ in \eqref{eq3.31}, we obtain the partial derivative formula for the function given in \eqref{eq2.1}, as following
	$$\frac{\partial^{(m+n)}}{\partial P^m \partial Q^n}	\left\{\mathfrak{B}_{(\phi,\psi)}^{(R,S,\eta,\xi)}(P,Q)\right\}$$
	\begin{equation}\label{eq3.32}
		=\int_{0}^{1}\ln^{m}(x)\ln^{n}(1-x)x^{P-I}(1-x)^{Q-I}\mathcal{E}_{(\phi,\psi)}\left(\frac{-R}{x^{\eta}}\right)\mathcal{E}_{(\phi,\psi)}\left(\frac{-S}{(1-x)^{\xi}}\right)~dx.
	\end{equation}

	\vspace{0.35cm}
	\section{\bf Numerical Representations}
	We illustrates the generalization presented in this paper with some examples, in this section matrices are from $\mathscr{M}_k.$ 
	\begin{example}
		For $\phi=1$, $\psi=2$, $\eta=0.5,$ $\xi=1$; $a=2$, $b=4$;
		$R= \begin{bmatrix}
			7 & 3\\
			3 & 4
		\end{bmatrix}$,
		$S= \begin{bmatrix}
			5 & 1\\
			1 & 3
		\end{bmatrix}$,
		$P= \begin{bmatrix}
			3 & 1\\
			1 & 2
		\end{bmatrix}$ and
		$Q= \begin{bmatrix}
			2 & 2\\
			2 & 5
		\end{bmatrix}$, then\\
		$x^{P-I}=x^{\begin{bmatrix}
				2 & 1\\
				1 & 1
		\end{bmatrix}}$, 
		$(1-x)^{Q-I}=(1-x)^{\begin{bmatrix}
				1 & 2\\
				2 & 4
		\end{bmatrix}}$,
		$\mathcal{E}_{(1,2)}{(-Rx^{-0.5})}=\large\sum_{k=0}^{\infty}{\frac{(-R)^k}{x^{0.5k}\xi(k +2)}}$ and\\
		$\mathcal{E}_{(1,2)}{(-S(1-x)^{-1})}=\large\sum_{k=0}^{\infty}{\frac{(-S)^k}{(1-x)^{k}\xi(k +2)}}$, thus we have\\
		$$\mathfrak{BL}_{(1,2)}^{(R,S,0.5,1)}(2,4;P,Q)$$
		$$=\int_{0}^{1}2^{1-x}4^{x}x^{\begin{bmatrix}
				2 & 1\\
				1 & 1
		\end{bmatrix}}(1-x)^{\begin{bmatrix}
				1 & 2\\
				2 & 4
		\end{bmatrix}}\mathcal{E}_{(1,2)}{(-Rx^{-0.5})}\mathcal{E}_{(1,2)}{(-S(1-x)^{-1})}dx~$$ 
		$$=\begin{bmatrix}
			0.0223 & -0.0369\\
			-0.0252 & 0.0423
		\end{bmatrix}.$$
		Similarly,\\
		$$\mathfrak{BL}_{(1,2)}^{(S,R,1,0.5)}(4,2;Q,P)=
		\begin{bmatrix}
			0.0334 & -0.0536\\
			-0.0192 & 0.0312
		\end{bmatrix}.$$
	\end{example}
	\begin{remark}
		From the above example we observe that, the symmetric property of the function \eqref{eq2.2} does not hold as: \\
		$$\mathfrak{BL}_{(1,2)}^{(R,S,0.5,1)}(2,4;P,Q)\ne\mathfrak{BL}_{(1,2)}^{(S,R,1,0.5)}(4,2;Q,P).$$
	\end{remark}	
	\begin{example}
		For $\phi=1$, $\psi=2$, $\eta=0.5$, $\xi=1$; $a=2$, $b=2$;
		$R= \begin{bmatrix}
			7 & 3\\
			3 & 4
		\end{bmatrix}$,
		$S= \begin{bmatrix}
			5 & 1\\
			1 & 3
		\end{bmatrix}$,
		$P= \begin{bmatrix}
			3 & 1\\
			1 & 2
		\end{bmatrix}$ and
		$Q= \begin{bmatrix}
			2 & 2\\
			2 & 5
		\end{bmatrix}$, then we have:
		$$\mathfrak{BL}_{(1,2)}^{(R,S,0.5,1)}(2,2;P,Q)=\begin{bmatrix}
			0.0154 & -0.0257\\
			-0.0182 & 0.0307
		\end{bmatrix},$$ and
		$$\mathfrak{BL}_{(1,2)}^{(S,R,1,0.5)}(2,2;Q,P)=\begin{bmatrix}
			0.0234 & -0.0378\\
			-0.0139 & 0.0227
		\end{bmatrix}.$$
	\end{example}
	\begin{remark}
		From the above example, we observe that, even if we take $a=b$, the symmetric property of the function \eqref{eq2.2} does not hold.
	\end{remark}	
	\begin{example}
		For $\phi=1$, $\psi=2$, $\eta=0.5$, $\xi=1$; $a=2$, $b=4$;
		$R= \begin{bmatrix}
			2 & 0\\
			0 & 2
		\end{bmatrix}$,
		$S= \begin{bmatrix}
			2 & 3\\
			-3 & 2
		\end{bmatrix}$,
		$P= \begin{bmatrix}
			0 & 1\\
			-1 & 0
		\end{bmatrix}$ and
		$Q= \begin{bmatrix}
			1 & 2\\
			-2 & 1
		\end{bmatrix}$, then we have \\
		$$\mathfrak{BL}_{(1,2)}^{(R,S,0.5,1)}(2,4;P,Q)=\begin{bmatrix}
			-0.2505 & 0.1049\\
			-0.1049 & -0.2505
		\end{bmatrix}=\mathfrak{BL}_{(1,2)}^{(S,R,1,0.5)}(4,2;Q,P).$$
	\end{example}
	
	\begin{example}
		For $\phi=1$, $\psi=2$, $\eta=0.5$, $\xi=1$; 
		$R= \begin{bmatrix}
			2 & 0\\
			0 & 2
		\end{bmatrix}$,
		$S= \begin{bmatrix}
			2 & 3\\
			-3 & 2
		\end{bmatrix}$,
		$P= \begin{bmatrix}
			0 & 1\\
			-1 & 0
		\end{bmatrix}$ and
		$Q= \begin{bmatrix}
			1 & 2\\
			-2 & 1
		\end{bmatrix}$, then we have
		$$\mathfrak{B}_{(1,2)}^{(R,S,0.5,1)}(P,Q)=\begin{bmatrix}
			-0.1015 & 0.0480\\
			-0.0480 & -0.1015
		\end{bmatrix}=\mathfrak{B}_{(1,2)}^{(S,R,1,0.5)}(P,Q).$$
	\end{example}
	
	\bigskip
	\begin{remark}
		After considering $P,Q,R$ and $S \in \mathscr{M}_{k} $ to be commutative in the above two examples, the symmetric property of the functions $\mathfrak{BL}_{(\phi,\psi)}^{(R,S,\eta,\xi)}(a,b;P,Q)$ and $\mathfrak{B}_{(\phi,\psi)}^{(R,S,\eta,\xi)}(P,Q)$, respectively does hold. Furthermore, the commutativity of the input elements ensures that the properties of the functions are consistent, making it easier to apply them in various mathematical contexts without loss of generality.
	\end{remark}	
	
	\section{\bf Graphical Representations}
	\subsection{\bf Graphical Comparison of the Generalized and Extended Beta-Logarithmic Matrix Function}
	We illustrate the graphical comparison of the generalized beta-logarithmic matrix function (GBLMF) presented in this paper with the previous extended beta-logarithmic matrix function (EBLMF) investigated by Alqarni \cite{alq}. Since, the comparison runs over two matrices, we use the infinity norms defined in \eqref{eq1.2}. All the matrices considered are $2\cross2$ in $\mathscr{M}_k$. The choice of used matrices are $P=\begin{bmatrix}
		1 & 0\\ 1 & 2
	\end{bmatrix}$ and $Q=\begin{bmatrix}
		1 & 1\\ 0 & 2
	\end{bmatrix}$ as in \cite{jod-cor} and multiple choices of $R$ and $S$ matrices have infinity-norm in $[0.1, 0.5]$. The matrix $R$ have been used in the exponential part of the function \cite{alq}.
\begin{figure}[H]
	\centering
	\includegraphics[width=0.7\linewidth]{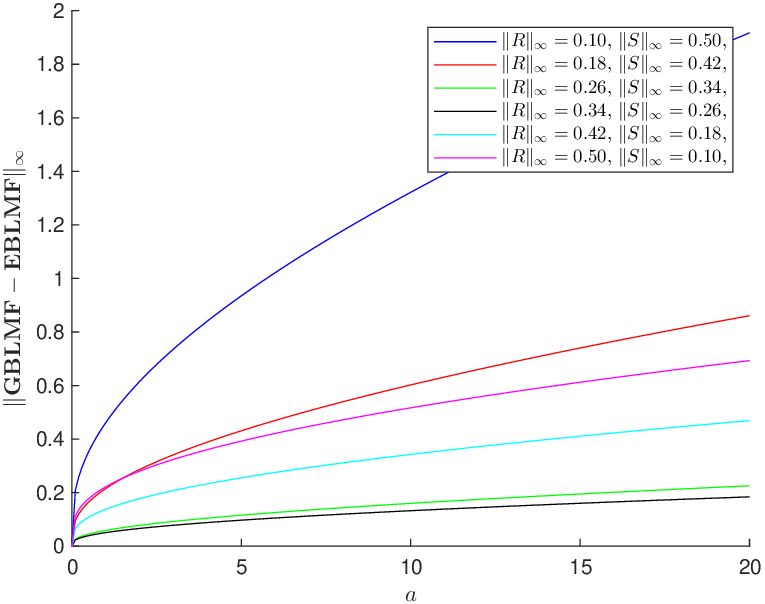}
	\caption{}
	\label{fig:figure1}
\end{figure}

	Figure 1 shows the graph of the difference between GBLMF and EBLMF against $a \in [0, 20]$ and $\phi=\psi=\eta=\xi=b=1$.  
\begin{figure}[H]
	\centering
	\includegraphics[width=0.7\linewidth]{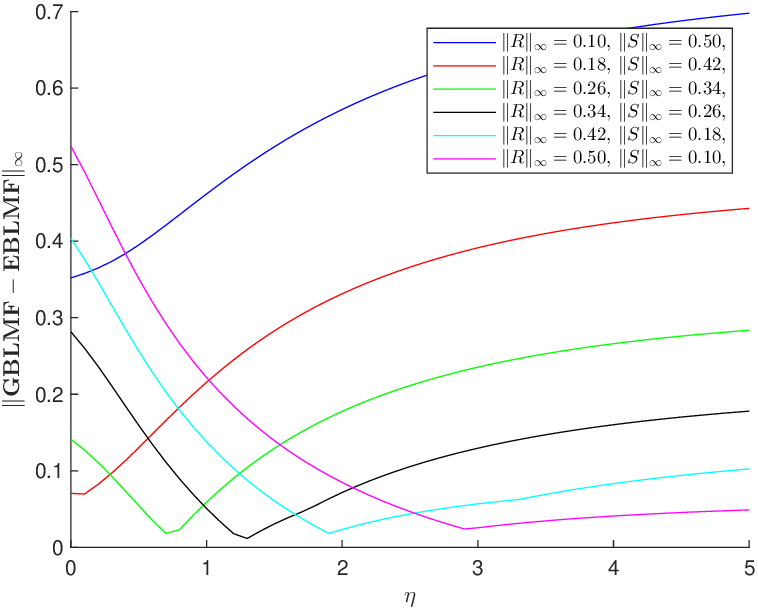}
	\caption{}
	\label{fig:figure2}
\end{figure}

Figure 2 shows the graph of the difference between GBLMF and EBLMF against $\eta \in [0, 5]$ and $\phi=\psi=\xi=a=b=1$.
\begin{figure}[H]
	\centering
	\includegraphics[width=0.7\linewidth]{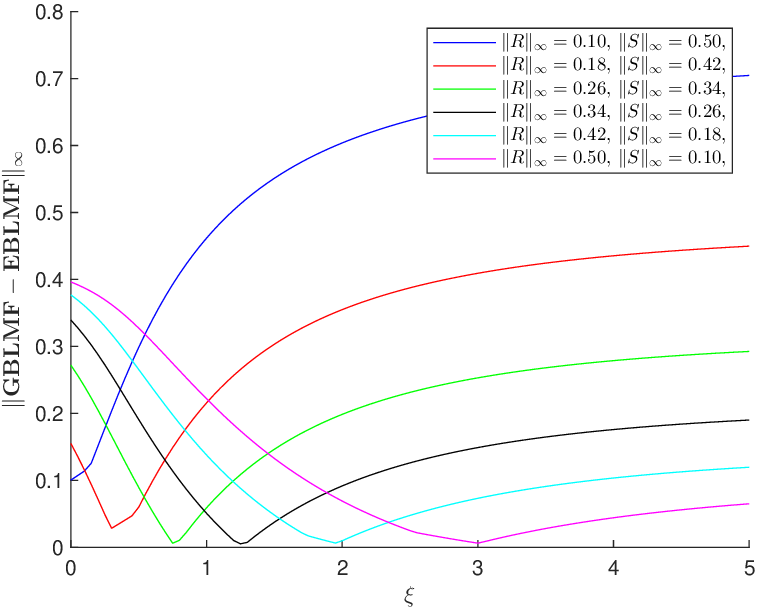}
	\caption{}
	\label{fig:figure3}
\end{figure}

Figure 3 shows the graph of the difference between GBLMF and EBLMF against $\xi \in [0, 5]$ and $\phi=\psi=\eta=a=b=1$.
	
\begin{figure}[H]
	\centering
	\includegraphics[width=0.7\linewidth]{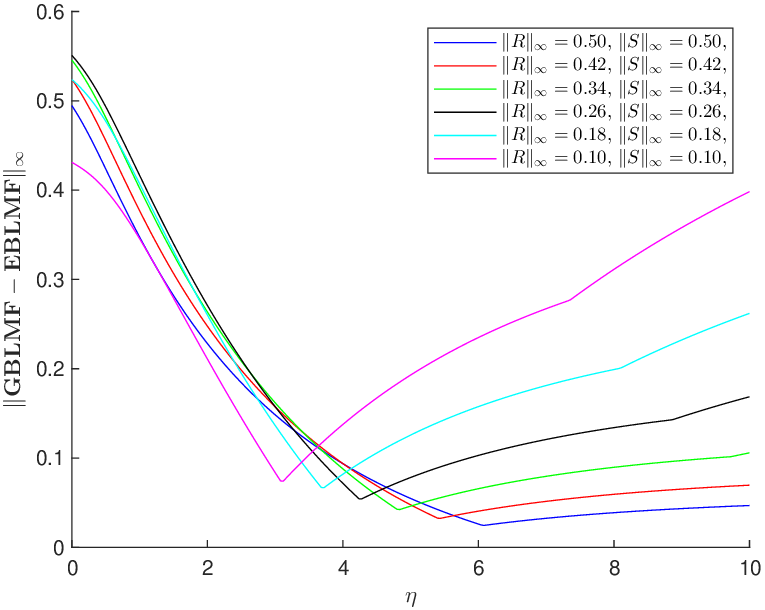}
	\caption{}
	\label{fig:figure4}
\end{figure}

	Figure 4 shows the graph of the difference between GBLMF and EBLMF against $\eta \in [0, 10]$. Both $R$ and $S$ approaches zero matrix. Here, $\phi=\xi=1$; $\psi=2$; $a=0.5$ and $b=2$.
\begin{figure}[H]
	\centering
	\includegraphics[width=0.7\linewidth]{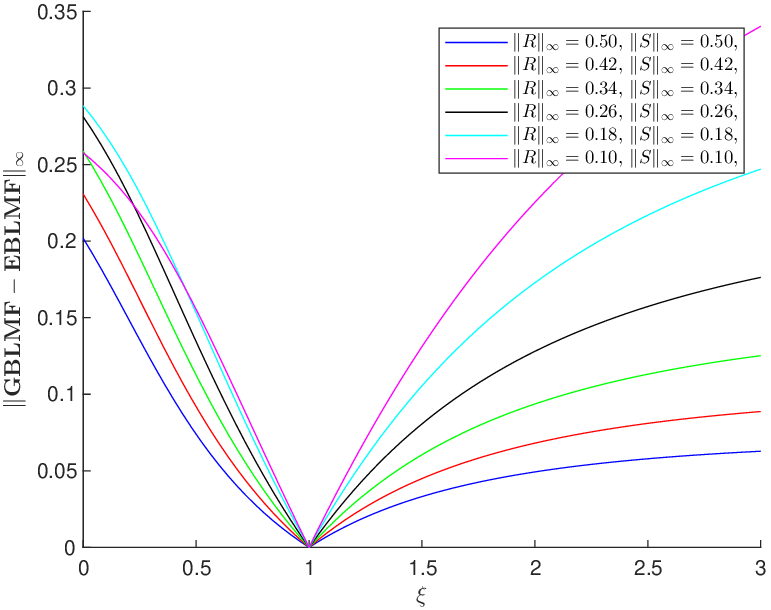}
	\caption{}
	\label{fig:figure5}
\end{figure}

	Figure 5 shows the graph of the difference between GBLMF and EBLMF against $\xi \in [0, 3]$. Both $R$ and $S$ approaches zero matrix. Here, $\phi=\psi=b=\eta=1$ and $a=0.5$.
	
	Here, the graphical comparison highlights the broader applicability and enhanced flexibility of the GBLMF compared to the EBLMF, demonstrating its potential for addressing a wider range of mathematical and applied problems. These advancements underscore the significance of the proposed function in extending the theoretical framework and its applicability to complex computational scenarios. Consequently, the GBLMF serves as a valuable tool for future research in mathematical analysis and related disciplines.
	\subsection{\bf Graphical Comparison of the Generalized Beta-Logarithmic and Classical Beta Matrix Function}
	We have illustrated the graphical comparison of the generalized beta-logarithmic matrix function (GBLMF) presented in this paper with the classical beta matrix function (CBMF) introduced by \cite{jod-cor}. The comparison highlights the enhanced versatility of the Generalized Beta-Logarithmic Matrix Function (GBLMF), which could be beneficial in applications requiring more tailored matrix transformations. These results suggest that the GBLMF may offer superior convergence properties or greater flexibility in numerical methods compared to the CBMF, especially in complex matrix analysis scenarios. The choice of the used matrices are $P=\begin{bmatrix}
		1 & 0\\ 1 & 2
	\end{bmatrix}$ and $Q=\begin{bmatrix}
		1 & 1\\ 0 & 2
	\end{bmatrix}$ as in \cite{jod-cor} and multiple choices of $R$ and $S$ matrices have infinity-norm in $[0.1, 0.5]$. 
\begin{figure}[H]
	\centering
	\includegraphics[width=0.7\linewidth]{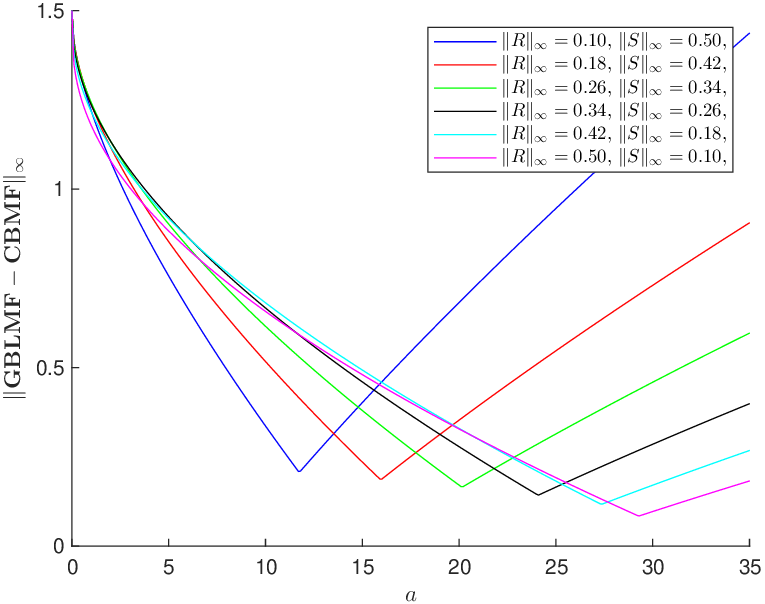}
	\caption{}
	\label{fig:figure6}
\end{figure}

	Figure 6 shows the graph of difference between GBLMF and CBMF against $a \in [0, 35]$. Here, $\phi=\psi=\eta=\xi=1$ and $b=1$.
	\begin{figure}[H]
		\centering
		\includegraphics[width=0.7\linewidth]{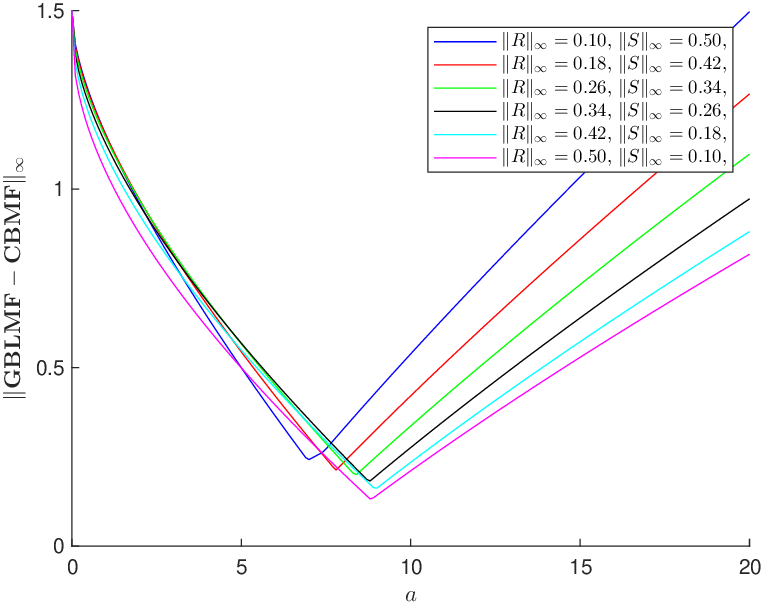}
		\caption{}
		\label{fig:figure7}
	\end{figure}

	Figure 7 shows the graph of difference between GBLMF and CBMF against $a \in [0, 20]$. Here, $\phi=\psi=\xi=1$, $\eta=0.5$ and $b=1$.
\begin{figure}[H]
	\centering
	\includegraphics[width=0.7\linewidth]{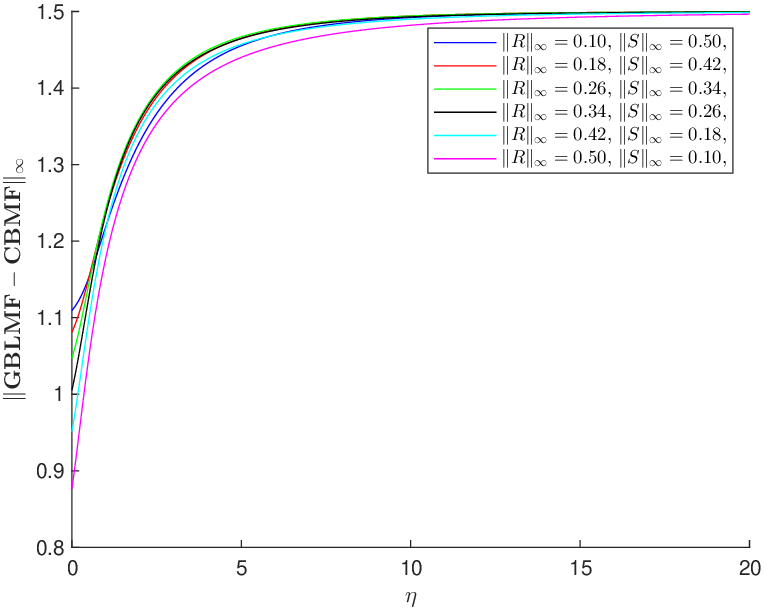}
	\caption{}
	\label{fig:figure8}
\end{figure}

	Figure 8 shows the graph of difference between GBLMF and CBMF against $\eta \in [0, 20]$. Here, $\phi=\psi=a=b=\xi=1$.

	\section{\bf Conclusion}
	In this article, we have thoroughly explored the generalized beta-logarithmic matrix function, which integrates the generalized beta matrix function with the logarithmic mean. Through detailed analysis, we have established a variety of essential properties of the function, including functional relations, inequalities, series expansions, integral representations, and partial derivative formulas. These findings enhance our understanding of the generalized beta-logarithmic matrix function and its potential applications in matrix analysis and related fields. Numerical examples and graphical representations have been provided to visually demonstrate the behavior of the function and to illustrate how it differs from both classical and previously studied versions of the beta matrix function. These comparisons highlight the novel aspects of the extended form, showcasing its advantages in terms of versatility and potential for broader applications.The extended beta-logarithmic matrix function opens up new avenues for further research, particularly in areas such as matrix inequalities, numerical methods, and applied mathematics. Future studies could explore additional applications in optimization problems, matrix theory, and other fields where matrix functions play a key role. In conclusion, this work provides a comprehensive theoretical framework for the extended beta-logarithmic matrix function, offering both a solid foundation for further theoretical investigation and practical tools for its application in various mathematical contexts.

		\vspace{0.75cm}
	\noindent\textbf{\textsf{Acknowledgments:}}
	The authors are very grateful to the anonymous referees for many valuable comments and suggestions which helped to improve
	the draft.
	
	\vspace{0.75cm}
	\noindent\textbf{\textsf{Author Contributions:}}
	All authors contributed equally to the present investigation.\\
	All authors read and approved the final manuscript.
	
	\vspace{0.2cm}
	\noindent\textbf{\textsf{Conflicts of Interest:}} The authors declare no conflicts of interest.
	
	\vspace{0.2cm}
	\noindent\textbf{\textsf{Funding (Financial Disclosure):}} This research received no external funding.

\end{document}